\documentclass[11pt, reqno]{amsart}
\usepackage{latexsym,amsmath,amssymb, amsthm, mathtools}
\usepackage{cases, verbatim}
\usepackage[colorlinks]{hyperref}
\usepackage{hypernat}
\usepackage{xcolor}
\usepackage{float}

\usepackage{mathscinet}

\usepackage[margin=1.5in]{geometry}

\newtheorem{thm}{Theorem}
\newtheorem{lem}[thm]{Lemma}
\newtheorem{prop}[thm]{Proposition}
\newtheorem{cor}[thm]{Corollary}

\theoremstyle{remark}
\newtheorem{rmk}[thm]{Remark}
\newtheorem{example}[thm]{Example}

\theoremstyle{definition}
\newtheorem{defi}[thm]{Definition}

\numberwithin{thm}{section} 
\numberwithin{equation}{section}

\makeatletter

\newcommand{\Rmnum}[1]{\expandafter\@slowromancap\romannumeral #1@}
\makeatother

\def\R{{\mathbb R}}

\def\N{{\mathcal N}}

\def\S{{\mathcal S}}

\newcommand{\Oba}{\overline{\Omega}}

\newcommand{\vep}{\varepsilon}
\newcommand{\ol}{\overline}

\newcommand{\tr}{\operatorname{tr}}

\newcommand{\bpm}{\begin{pmatrix}}
\newcommand{\epm}{\end{pmatrix}}
\newcommand{\la}{\langle}
\newcommand{\ra}{\rangle}

\newcommand{\beq}{\begin{equation}}
\newcommand{\eeq}{\end{equation}}

plus 6pt minus 12pt
\title[Liouville-type theorem with boundary degeneracy]{\protect{Liouville-type theorems for fully nonlinear elliptic and parabolic equations with boundary degeneracy}}

\author{Qing Liu}
\address[Qing Liu]{Geometric Partial Differential Equations Unit, Okinawa Institute of Science and Technology 
Graduate University, 1919-1 Tancha, Onna-son, Kunigami-gun, 
Okinawa, 904-0495, Japan}
\email{qing.liu@oist.jp} 

\author{Erbol Zhanpeisov}
\address[Erbol Zhanpeisov]{Geometric Partial Differential Equations Unit, Okinawa Institute of Science and Technology 
Graduate University, 1919-1 Tancha, Onna-son, Kunigami-gun, 
Okinawa, 904-0495, Japan}
\email{erbol.zhanpeisov@oist.jp}

\date{\today}

\begin{document}

\begin{abstract}
We study a class of fully nonlinear boundary-degenerate elliptic equations, for which we prove that $u \equiv 0$ is the only solution. Although no boundary conditions are posed together with the equations, we show that the operator degeneracy actually generates an implicit boundary condition. Under appropriate assumptions on the degeneracy rate and regularity of the operator, we then prove that there exist no bounded solutions other than the trivial one. Our method is based on the arguments for uniqueness of viscosity solutions to state constraint problems for Hamilton-Jacobi equations. We obtain similar results for fully nonlinear degenerate parabolic equations. Several concrete examples of equations that satisfy the assumptions are also given. 
\end{abstract}


\maketitle

\section{Introduction}
In this paper, we present Liouville-type results for a class of fully nonlinear elliptic or parabolic equations with ellipticity of the operator degenerate on the boundary. Our work is motivated by a recent paper \cite{BiPu} on nonexistence of nontrivial solutions to the boundary value problem 
\begin{equation}\label{linear eq}
Q(x) u-\Delta u-\la b(x), \nabla u\ra =0 \quad \text{in $\Omega$, }
\end{equation}
where $\Omega\subset \R^n$ is assumed to be a bounded domain, and $Q\in C(\Omega), b\in C^1(\Omega; \R^n)$ are given functions. The authors of \cite{BiPu} assume that the coefficient $Q$ is positive in $\Omega$ and singular near the boundary in the sense that $Q(x)\to \infty$ as $d(x)\to 0$, where $d(x)$ denotes the distance from $x\in \Oba$ to $\partial \Omega$. The drift term $b$ is also allowed to be unbounded, that is, it may happen that $|b(x)|\to \infty$ as $d(x)\to 0$ as well.  

Note that, due to the positivity of $Q$, \eqref{linear eq} can be equivalently transformed into a degenerate elliptic equation
\begin{equation}\label{linear eq2}
u-\frac{1}{Q(x)} \Delta u-\frac{1}{Q(x)}\la b(x), \nabla u\ra=0 \quad \text{in $\Omega$}
\end{equation}
without changing solutions. It is clear that $u\equiv 0$ is a solution to \eqref{linear eq}. The main result of \cite{BiPu} states that it is actually the only bounded solution under appropriate regularity assumptions on $\Omega$ and several key conditions on the growth rates of $Q$ and $b$ near $\partial \Omega$.    

This Liouville-type result in \cite{BiPu} should also be regarded as a uniqueness result for the associated boundary degenerate elliptic equations. Interestingly, the solutions are still unique despite the absence of boundary conditions. We also refer to \cite{PoPuTe, PuTe} for other related results on the well-posedness of such a type of linear degenerate elliptic and parabolic equations. 
In this work, we pursue similar unique results for general fully nonlinear elliptic and parabolic equations. Let us first consider the elliptic case
\begin{equation}\label{nonlinear eq}
u+F(x, \nabla u, \nabla^2 u)=0\quad \text{in $\Omega$,} 
\end{equation}
where $F: \Omega \times \R^n\times \S^n\to \R$ satisfies the set of assumptions (F1)--(F4) below. Here $\S^n$ stands for the set of all $n\times n$ symmetric matrices. Similar results for parabolic equations will be mentioned later. It is worth pointing out that our PDE setting in this work is different from many other Liouville-type results for viscosity solutions of fully nonlinear equations, for which we refer to \cite{Cap, CuLe, CaCu, BiDe1, ArSi1, ArSi2, ChFe, LuZh, BaCe, BiGaLe, BiDeLe, BaGo1}
 and references therein. 

Our general results apply to a possibly unbounded domain $\Omega$. Concerning its regularity, we impose the following assumption. 
\begin{enumerate}
\item[(A)] For any $x_0\in \partial \Omega$ and $s>0$ small, there exist  $\nu\in C(\Oba\cap B_s(x_0) ; \R^n)$ and $c>0$ small such that 
\begin{equation}\label{bdry reg}
B_{c\tau}(y+\tau \nu(y))\subset \Omega \quad \text{ for all $y\in \Oba\cap B_s(x_0)$ and $0<\tau\leq c$.} 
\end{equation}
\end{enumerate}
Note that any domain of class $C^1$ satisfies this condition. When $\Omega$ is a $C^1$ domain, one can obtain $\nu$ by extending continuously the inward unit normal $\nu$ to $\partial \Omega$ around $x_0$ to $\Oba\cap B_s(x_0)$. We remark that  (A) is a standard assumption for the domain in the study of the so-called state constraint problems; see \cite{So1, So2, BCBook} etc. Our approach is actually closely related to optimal control with state constraints. More details about the connection will be explained later. 

Let us now list the assumptions on the fully nonlinear operator $F$. Throughout this paper, $d(x)$ denotes the distance from $x\in \Oba$ to $\partial \Omega$.
\begin{enumerate}
\item[(F1)] (Continuity) $F\in C(\Omega\times \R^n\times \S^n)$. 
\item[(F2)] (Ellipticity) $F$ is elliptic, i.e., $F(x, p, X)\leq F(x, p, Y)$ for all  $x\in \Omega$, $p\in \R^n$, and $X, Y\in \S^n$ satisfying $X\geq Y$.
\item[(F3)] (Degeneracy) For any modulus of continuity $\omega$ (that is, a nonnegative increasing $\omega\in C([0, \infty))$ with $\omega(0)=0$),  we have 
\begin{equation}\label{degeneracy}
\begin{aligned}
\sup \Big\{|F(x, p, X)|: \  & x\in \Omega,  \ d(x)\leq \delta, \\
& \quad |p|\leq \omega(\delta)/\delta, |X|\leq \omega(\delta)/\delta^2 \Big\}\to  0\quad \text{as $\delta\to 0$}.
\end{aligned}
\end{equation}
\item[(F4)] (Zero solution) $F$ satisfies  
\begin{equation}\label{near zero}
\sup_{x\in \Omega}|F(x, p, X)|\to 0\quad \text{as $|p|, |X|\to 0$.}
\end{equation}
\end{enumerate} 
The condition (F4) ensures that $u\equiv 0$ is a solution of \eqref{nonlinear eq}. When $\Omega$ is assumed to be bounded, we can relax this condition by simply assuming $F(x, 0, 0)=0$ for all $x\in \Omega$.

The assumption (F3) can be viewed as a degeneracy condition of $F$. It actually describes the singularity of coefficients as $Q$ in \eqref{linear eq}. It is clear that the linear operator in \eqref{linear eq}, written in the form \eqref{linear eq2}, corresponds to
\[
F(x, p, X)=-\frac{1}{Q(x)} \tr X-\frac{1}{Q(x)}\la b(x), p\ra, \quad x\in \Omega, p\in \R^n, X\in \S^n
\]
and satisfies (F1)--(F4). In this case, we also have (F3) provided that for any modulus of continuity $\omega$, 
\[
{\omega(d(x))\over Q(x)d(x)^2}+{|b(x)|\omega(d(x))\over Q(x)d(x)}\to 0\quad \text{ uniformly as $d(x)\to 0$.}
\]
This indeed holds true if there exists $C>0$ such that 
\[
Q(x)\geq {C\over d(x)^2}, \quad |b(x)|\leq {C\over d(x)^\alpha} \quad \text{with $0<\alpha\leq 1$}
\]
for all $x\in \Omega$ near $\partial \Omega$, which is consistent with the assumptions in \cite{BiPu}. The condition turns out to be optimal as shown in the examples in \cite{BiPu}.

Given the fully nonlinear structure of \eqref{nonlinear eq}, the notion of viscosity solution becomes a natural choice for us to understand properties of the equation. Our focus is to show the nonexistence of bounded viscosity solutions $u$ of \eqref{nonlinear eq} other than the trivial one $u\equiv 0$. It turns out that we can obtain such a result for any operator $F$ satisfying (F1)--(F4) if we only consider bounded solutions that can be extended to a continuous function in $\Oba$.  

\begin{thm}[Liouville-type theorem for boundary continuous solutions]\label{thm main}
Let $\Omega\subset \R^n$ be a domain satisfying (A). Assume that (F1)--(F4) hold. If $u\in C(\Omega)$ is a viscosity solution to \eqref{nonlinear eq} that can be continuously extended to $\partial \Omega$, then $u\equiv 0$ in $\Omega$. 
\end{thm}

The proof in \cite{BiPu} for the linear equation \eqref{linear eq} is based on careful a priori estimates via the associated parabolic problem. Since the strong nonlinearity and non-divergence nature of \eqref{nonlinear eq} no longer enable us to directly establish integral estimates, we adopt a different approach, looking into the boundary behavior of solutions via comparison with appropriate smooth barrier functions.  
Our primary observation about this problem is that if $u$ is a solution satisfying $u\neq 0$ in a neighborhood of a boundary point $x_0\in \partial \Omega$, then $u$ cannot be extended to a $C^2$ function around $x_0$, for otherwise (F3) yields a contradiction that 
\[
0=\lim_{x\to x_0}  \{u(x)+ F(x, \nabla u(x), \nabla^2 u(x))\}= u(x_0)\neq 0.
\] 

In the framework of viscosity solutions, where $C^2$ regularity of solutions is usually not expected, one can deal with the degeneracy condition (F3) via suitable test functions, which  are used to handle the aforementioned state constraint problems. In fact, we adapt the uniqueness proof in \cite{So1, So2} to show that if any viscosity solution $u\in C(\Omega)$ of \eqref{nonlinear eq} can be continuously extended to $\Omega \cup \{x_0\}$ (for some $x_0 \in \partial \Omega$), then $u(x_0)=0$ holds. A more precise statement of this result is given in Proposition \ref{prop singular bdry}.  

Our interpretation above suggests that, although no explicit boundary conditions are imposed in this problem, the degeneracy of the equation \eqref{nonlinear eq} leads us to a Dirichlet boundary condition 
\begin{equation}\label{dirichlet}
u=0 \quad \text{on $\partial \Omega$};
\end{equation}
 see Proposition \ref{prop dirichlet}.  Once the homogeneous Dirichlet condition is obtained, the conclusion $u\equiv 0$ in $\Omega$ then follows easily from the standard maximum principle.

Our result in Theorem \ref{thm main} applies to a very broad class of fully nonlinear elliptic equations. The extendability of $u$ to a function in $C(\Oba)$ plays a key role. One sufficient condition that guarantees the continuous extension is uniform continuity of $u$ in $\Omega$, but such regularity for viscosity solutions of degenerate elliptic equations is not entirely clear and needs to be further investigated. 
 We thus turn to a more specific class of equations that enjoy the Liouville-type property without requiring the extendability condition on solutions. To this end, we impose (F5) below on $F$ in addition to (F1)--(F4). 
  
\begin{enumerate}
\item[(F5)] For any $M>0$, there exist $a_F, L>0$ and $\alpha\in (0, 1], \beta\in (0, 1)$ possibly depending on $M$ such that 
\begin{equation}\label{uc-cond}
F(y, q, Y) -  F(x, p, X) \leq  \beta a_F |x-y|^\alpha+L\delta
\end{equation}
holds for any $\delta>0$, $x, y\in \Omega$ with $x\neq y$, $p, q\in \R^n$ and $X, Y\in \S^n$ provided that
\begin{equation}\label{uc-bounds}
F(x, p, X)\leq M,\ \ F(y, q, Y)\geq -M, 
\end{equation}
\begin{equation}\label{a5-1}
\begin{aligned}
&|p-\alpha a_F |x-y|^{\alpha-2}(x-y)| \leq \delta d(x)^{-1}, \\
&|q- \alpha a_F |x-y|^{\alpha-2}(x-y)|\leq  \delta d(y)^{-1},  
\end{aligned}
\end{equation}
as well as 
\begin{equation}\label{a5-2}
 \begin{pmatrix} X & 0\\ 0& -Y\end{pmatrix}\leq \begin{pmatrix} A & -A \\ -A & A\end{pmatrix}+\delta \begin{pmatrix} {d(x)^{-2}} I_n & 0\\ 0 & {d(y)^{-2}} I_n \end{pmatrix}+ \delta I_{2n} 
\end{equation}
hold with $A=A(x, y)\in \S^n$ satisfying 
\begin{equation}\label{uc-mod-rev}
A\leq \alpha a_F |x-y|^{\alpha-2} I_n. 
\end{equation}
\end{enumerate}

Our second Liouville-type result is as follows. 
\begin{thm}[Liouville-type theorem for continuous solutions]\label{thm main2}
Let $\Omega\subset \R^n$ be a bounded domain of $C^2$ class. Assume that (F1)--(F5) hold. If $u\in C(\Omega)$ is a bounded viscosity solution to \eqref{nonlinear eq}, then $u\equiv 0$ in $\Omega$. 
\end{thm}

It turns out that under the assumptions in Theorem \ref{thm main2}, we can derive (global) H\"older continuity in a bounded and $C^2$ domain $\Omega$ for any bounded continuous solution. The additional boundedness and regularity assumptions on $\Omega$ are needed for us to have uniform boundedness and $C^2$ regularity of the distance function $d$. 
This result is shown in Proposition \ref{prop uc}, where we essentially develop the arguments of \cite[Theorem 3.1]{I84}, \cite[Theorem VII.1]{IL} and \cite[Chapter II, Proposition 4.2]{BCBook} for Lipschitz or H\"older regularity of viscosity solutions to Hamilton-Jacobi equations. A direct application of Theorem \ref{thm main} then immediately yields Theorem~\ref{thm main2}.  

Let us give more explanations about (F5). Roughly speaking, if we take $\delta=0$ in (F5) for the moment, then \eqref{uc-cond}, together with other conditions in (F5), basically amounts to saying that 
\[
\psi(x, y):= a_F|x-y|^\alpha
\]
is a strict supersolution (except at $x=y$) of 
\[
\psi(x, y)+F(x, \nabla_x\psi, \nabla_x^2\psi)-F(y, \nabla_y \psi, \nabla_y^2\psi)=0.
\]
Since $u(x)-u(y)$ is a subsolution of the same equation, by comparison principle it is thus reasonable to expect that $u(x)-u(y)\leq \psi(x, y)$ for $x, y\in \Omega$, which is indeed proved rigorously in Proposition~\ref{prop uc} in the framework of viscosity solutions. The terms with the parameter $\delta$ are related to an extra penalization near the boundary. See Lemma \ref{lem f5} for a sufficient condition of (F5), which is useful for us to understand concrete examples in Section~\ref{sec:example}. 

Our perspective of generating boundary conditions from the degeneracy can be naturally extended to fully nonlinear parabolic equations. Let $u\in C(\Oba\times (0,T))$ be a viscosity solution of 
\begin{equation}\label{parabolic eq}
    u_t+H(x, u, \nabla u, \nabla^2 u)=0\quad \text{in $\Omega\times (0, T)$}
\end{equation}
where $T>0$ is a given constant and $H$ is a continuous elliptic operator. If 
\begin{equation}\label{degeneracy general2}
\begin{aligned}
\sup \Big\{|H(x, r, p, X)-G(y, r)|:& \  x\in \Omega,\ y\in \partial \Omega,  \ |x-y|\leq \delta, \ r\in \R, \\
&  |p|\leq \omega(\delta)/\delta, |X|\leq \omega(\delta)/\delta^2 \Big\}\to  0\quad \text{as $\delta\to 0$}
\end{aligned}
\end{equation}
for some $G\in C(\partial \Omega\times \R)$, then we can derive a nonlinear dynamic boundary condition $u_t+G(x, u)=0$ upon tests of $C^2$ functions on $ \partial \Omega\times (0, T)$. See Proposition \ref{prop general bdry} for more details. An analogous Liouville-type result in this general case is obtained in Theorem \ref{thm liouville general}. In Theorem \ref{thm main3}, we also present a parabolic version of Theorem \ref{thm main2}. As in the elliptic case, in order for the equation to admit only trivial solutions, the exponent range for boundary degeneracy with coefficient $d(x)^\mu$ of the second order operator is still $\mu\geq 2$. Related uniqueness results for linear boundary-degenerate parabolic equations are established in \cite{IshMu1, IshMu2} by adopting intrinsic metrics associated to the equations. Solvability of linear boundary-degenerate parabolic equations in a half space with boundary degree $\mu<2$ and prescribed Dirichlet condition has been recently studied in \cite{DPT1, DPT2} using different methods.

Section \ref{sec:example} is devoted to more precise applications of our results. For $C^2$ bounded domains, we provide several concrete examples that satisfy the assumptions in Theorem
~\ref{thm main2}, especially the condition (F5). Our examples include several important types of elliptic operators such as the Laplacian, the viscous Hamilton-Jacobi-Isaacs operator and the Pucci operator, all carried with boundary degenerate coefficients. We can also apply Theorem~\ref{thm main2} to a class of first order Hamilton-Jacobi equations, where the critical value of the exponent $\mu$ changes in accordance with the growth of the Hamiltonian with respect to the momentum variable. 

Let us conclude the introduction with a comparison between our results and those known on the state constraint problems in \cite{So1, So2, IK3, Kat, BaBu, BaRo, ILo} as well as other boundary value problems. While our approach shows a strong connection to the uniqueness arguments of the state constraint problems, we would like to highlight two major differences. 

(1) The purpose of this work is different. The state constraint boundary condition is usually posed to model state-constrained control processes in the aforementioned papers, among which \cite{Kat, BaBu, BaRo, ILo} actually also allow boundary-degenerate ellipticity. We in contrast do not impose any boundary condition in our PDE setting. The boundary condition derived in our analysis arises naturally from the operator degeneracy. Our estimate for the uniform continuity to prove Theorem~\ref{thm main2} is also obtained without using any boundary condition. We refer to \cite{FePo1, FePo2, DP1, DP2, DP3, DPT1, DPT2} for well-posedness for linear degenerate elliptic or parabolic equations with prescribed Dirichlet boundary conditions. Our setting is different from theirs as well. 

 (2) In addition, our results have a broader scope of PDE applications. Since the continuity of solutions up to the boundary for state constraint problems is obtained from the optimal control interpretation, it requires the equation to be of the convex Hamilton-Jacobi type. We however apply PDE methods to prove this regularity, and thus we can consider more general degenerate elliptic or parabolic equations, including first order nonconvex Hamilton-Jacobi equations. We do not need the optimal control formulation for the existence of solutions, as $u\equiv 0$ is always a solution by (F4). 

\subsection*{Acknowledgments}
The authors would like to thank Daniel~Hauer, Kazuhiro~Ishige, Hiroyoshi~Mitake and Hung~V.~Tran for helpful discussions and references. The authors are also grateful to the anonymous referee for valuable comments. 
The work of QL was supported by JSPS Grant-in-Aid for Scientific 
Research (No.~19K03574, No.~22K03396).  
The work of EZ was supported by JSPS Grant-in-Aid for Early-Career Scientists (No.~23K13005) and Fostering Joint International Research (No.~22KK0035).

\section{Implicit boundary conditions}
Let us go over the notion of viscosity solution.  
See \cite{CIL} for a comprehensive introduction to the theory of viscosity solutions. 
Hereafter we denote by $USC(S)$ and $LSC(S)$ respectively the classes of upper and lower semicontinuous functions in a set $S\subset \R^n$.

\begin{defi}\label{def viscosity}
A locally bounded function $u\in USC(\Omega)$ (resp., $LSC(\Omega)$) is called a viscosity subsolution (resp., supersolution) of \eqref{nonlinear eq} if whenever there exist $x_0\in \Omega$ and $\varphi\in C^2(\Omega)$ such that $u-\varphi$ attains a local maximum (resp., minimum) at $x_0$, we have 
\[
u(x_0)+F(x_0, \nabla \varphi(x_0), \nabla^2 \varphi(x_0))\leq 0
\]
\[
\left(\text{resp.,}\ \ u(x_0)+F(x_0, \nabla \varphi(x_0), \nabla^2 \varphi(x_0))\geq 0 \right).
\]
A function $u\in C(\Omega)$ is called a viscosity solution of \eqref{nonlinear eq} if it is both a viscosity subsolution and a viscosity supersolution of \eqref{nonlinear eq}.
\end{defi}

\begin{rmk}\label{rmk jets}
It is well known that viscosity tests can be characterized via the so-called semijets $J^{2, \pm} u$. More precisely, we have
\[
\begin{aligned}
&J^{2, +} u(x)=\{(\nabla \varphi(x), \nabla^2\varphi(x))\in \R^n\times \S^n: \text{ $u-\varphi$ attains a local maximum at $x$}\}\\
&=\left\{(p, X)\in \R^n\times \S^n: u(y)\leq u(x)+\la p, y-x \ra+ {1\over 2}\la X (y-x), y-x\ra+o(|y-x|^2)\right\}
\end{aligned}
\]
for any $x\in \Omega$ and $y$ near $x$, and $J^{2, -}u=-J^{2, +} (-u)$. Therefore, a function $u\in USC(\Omega)$ (resp., $u\in LSC(\Omega)$) is a viscosity subsolution (resp., supersolution) of \eqref{nonlinear eq} if and only if 
\[
u(x_0)+F(x_0, p_0, X_0)\leq 0 \quad \text{(resp., $\geq 0$)}
\]
holds for all $(p_0, X_0)\in J^{2, +} u(x_0)$ (resp., $(p_0, X_0)\in J^{2, -}u(x_0)$). Furthermore, it is known that, under the assumption (F1), one can replace $J^{2, \pm} u$ by their closures $\ol{J}^{2, \pm} u$ in this definition. See \cite{CIL} for more details. 
\end{rmk}
For our convenience, when referring to viscosity solutions later, we shall always omit the term ``viscosity''. 

Let us next discuss the boundary condition arising from (F3). We denote by $B_r(x)$ the open ball in $\R^n$ centered at $x$ with radius $r>0$. Moreover, for a bounded function $u$ in a set $S\subset \R^n$, we denote by $u^\ast$ and by $u_\ast$ its upper and lower semicontinuous envelopes respectively in the closure $\ol{S}$ of $S$; namely, for any $x\in \ol{S}$, we take
\[
\begin{aligned}
u^\ast(x)&= \lim_{\delta\to 0} \sup\{u(y): y\in S, \ |y-x|\leq \delta\},\\
u_\ast(x)&= \lim_{\delta\to 0} \inf\{u(y): y\in S, \ |y-x|\leq \delta\}.
\end{aligned}
\]

\begin{prop}[Implicit boundary conditions]\label{prop singular bdry}
Let $\Omega\subset \R^n$ be a domain satisfying (A). Assume that (F1), (F2) hold.  Let $x_0\in \partial \Omega$. 
\begin{enumerate}
\item[1)] Assume that for any modulus of continuity $\omega$,
\begin{equation}\label{degeneracy sub}
\qquad \lim_{\delta\to 0+}\inf \Big\{F(x, p, X): \  x\in \Omega,  \ |x-x_0|\leq \delta, 
|p|\leq \omega(\delta)/\delta, |X|\leq \omega(\delta)/\delta^2 \Big\}\geq  0.
\end{equation}
If $u\in USC(\Omega)$ is a bounded subsolution to \eqref{nonlinear eq} and $u^\ast$ is continuous at $x_0\in \partial \Omega$, then $u^\ast(x_0)\leq 0$.
\item[2)] Assume that for any modulus of continuity $\omega$, 
\begin{equation}\label{degeneracy sup}
\qquad \lim_{\delta\to 0+}\sup \Big\{F(x, p, X): \   x\in \Omega,  \ |x-x_0|\leq \delta, 
|p|\leq \omega(\delta)/\delta, |X|\leq \omega(\delta)/\delta^2 \Big\}\leq  0.
\end{equation}
If $u\in LSC(\Omega)$ is a bounded supersolution to \eqref{nonlinear eq} and $u_\ast$ is continuous at $x_0\in\partial \Omega$, then $u_\ast(x_0)\geq 0$.   
\end{enumerate}
In particular, if we assume (F3), then any solution $u\in C(\Omega)$ of \eqref{nonlinear eq} that is continuous at $x_0\in \partial \Omega$ satisfies $u(x_0)=0$.
\end{prop}
\begin{proof}
We only show the statement 1) of the proposition. The statement 2) can be proved using a symmetric argument.
 
It is clear that $u^\ast\in USC(\Oba)$ and $u^\ast(x)=u(x)$ for $x \in \Omega$. Let $x_0\in \partial \Omega$.  
We take $B_r(x_0)$ with $r>0$ sufficiently small so that (A) holds for $s=r$. Let $\N_r(x_0)=B_r(x_0)\cap \Omega$.  

Let $\nu$ be the function in $\ol{\N_r(x_0)}$ appearing in \eqref{bdry reg} of (A). 
For $\vep>0$, let 
\[
\Phi_\vep(x, y)=u^\ast(x)-\left|{x-y\over \vep}-\nu(y)\right|^4-\frac{|y-x_0|^2}{\vep^2},  \quad \text{for }x, y\in \ol{\N_r(x_0)}.
\]
By the upper semicontinuity of $\Phi_\vep$ in $\ol{\N_r(x_0)}\times \ol{\N_r(x_0)}$, we see that $\Phi_\vep$ attains a maximum in $\ol{\N_r(x_0)}\times \ol{\N_r(x_0)}$ for $\vep>0$ sufficiently small. Suppose that the maximum is attained at $(x_\vep, y_\vep)\in \ol{\N_r(x_0)}\times \ol{\N_r(x_0)}$. It is clear that
\[
\Phi_\vep(x_\vep, y_\vep)\geq \Phi_\vep(x_0+\vep \nu(x_0), x_0)=u^\ast(x_0+\vep\nu(x_0)),
\]
which implies that 
\begin{equation}\label{singular prop eq1}
\left|{x_\vep-y_\vep\over \vep}-\nu(y_\vep)\right|^4+{|y_\vep-x_0|^2\over \vep^2}\leq u^\ast(x_\vep)-u^\ast(x_0+\vep\nu(x_0)).
\end{equation}
Then, there exists $C>0$ such that $|{x_\vep-y_\vep}|\leq C\vep$ and $|y_\vep-x_0|\leq C\vep$ for all $\vep>0$ small, which further implies
\begin{equation}\label{singular prop eq2}
|x_\vep-x_0|\leq 2C\vep. 
\end{equation}
We thus have $x_\vep, y_\vep\to x_0$ as $\vep\to 0$. 
It then follows from \eqref{singular prop eq1} and the continuity of $u^\ast$ at $x_0$ that  
\[
\limsup_{\vep\to 0}\left|{x_\vep-y_\vep\over \vep}-\nu(y_\vep)\right|^4\leq \limsup_{\vep\to 0}\big(u^\ast(x_\vep)-u^\ast(x_0+\vep\nu(x_0))\big)=0, 
\]
which yields 
\[
\left|{x_\vep-y_\vep\over \vep}-\nu(y_\vep)\right|^4\to 0\quad \text{as $\vep\to 0$.}
\]
In particular, we get $x_\vep-(y_\vep+\vep \nu(y_\vep))=o(\vep)$, and therefore by \eqref{bdry reg} in (A), $x_\vep\in \Omega$ when $\vep>0$ is sufficiently small. 

Noticing that 
\[
\psi_\vep(x):= \left|{x-y_\vep\over \vep}-\nu(y_\vep)\right|^4+{|y_\vep-x_0|^2\over \vep^2}
\]
plays the role of a test function for the subsolution $u$ at $x=x_\vep$, i.e., $u-\psi_\vep$ attains a maximum at $x_\vep$, we apply the definition of subsolutions to get
\begin{equation}\label{singular prop eq3}
u(x_\vep)+F(x_\vep, \nabla \psi_\vep(x_\vep), \nabla^2 \psi_\vep(x_\vep))\leq 0.
\end{equation}
By direct computations, it is easily seen that
\[
\vep |\nabla \psi_\vep(x_\vep)|\to 0, \quad  \vep^2 |\nabla^2 \psi_\vep(x_\vep)|\to 0\quad \text{ as $\vep\to 0$.}
\]
Therefore, combining them with \eqref{singular prop eq2}, we can adopt \eqref{degeneracy sub} to deduce
\[
\liminf_{\vep\to 0}F(x_\vep, \nabla \psi_\vep(x_\vep), \nabla^2 \psi_\vep(x_\vep))\ge 0.
\]
By \eqref{singular prop eq3} and again the continuity of $u^\ast$ at $x_0$, this yields $u^\ast(x_0)\leq 0$. 

Note that the assumption (F3) implies the conditions \eqref{degeneracy sub} and \eqref{degeneracy sup}. Hence, we obtain both properties in 1) and 2) under (F3) for a solution continuous at $x_0 \in \partial \Omega$.
\end{proof}

\begin{rmk}
We do not need the boundedness assumption on $\Omega$. Our proof is based on the existence of a continuous function $\nu$ and $c>0$ satisfying \eqref{bdry reg} for $y\in \Oba$ near $x_0\in \partial \Omega$. 
\end{rmk}

One can use the same arguments of Proposition \ref{prop singular bdry} to show that, for a bounded subsolution $u\in USC(\Omega)$ (resp., supersolution $u\in LSC(\Omega)$) of 
\begin{equation}\label{inhomo dirichlet}    
u(x)+F(x, \nabla u, \nabla^2 u)=f(x)
\end{equation}
with $f\in C(\Oba)$, 
we have 
\[
u^\ast(x_0)\leq f(x_0)\quad \left(\text{resp.,}\quad u_\ast(x_0)\geq f(x_0) \right)
\]
provided that $u^\ast\in USC(\Oba)$ (resp., $u_\ast\in LSC(\Oba)$) is continuous at $x_0\in \partial \Omega$.  

A related result in the setting of parabolic equations is presented in Section \ref{sec:general}.

\section{Liouville-type theorems}

\subsection{Nonexistence of nontrivial bounded solutions in $C(\Oba)$}

Let us now prove that \eqref{nonlinear eq} does not have nontrivial bounded solutions that are continuous up to $\partial \Omega$. It follows from Proposition \ref{prop singular bdry} that any bounded solution that has a continuous extension to $\Oba$ must satisfy the homogeneous Dirichlet boundary condition \eqref{dirichlet}. 

\begin{prop}[Homogeneous Dirichlet condition]\label{prop dirichlet}
Let $\Omega\subset \R^n$ be a domain satisfying (A). Assume that (F1)--(F3) hold. Let $u\in C(\Omega)$ be a bounded solution to \eqref{nonlinear eq} that can be extended to $u\in C(\overline{\Omega})$. Then $u=0$ on $\partial \Omega$. 
\end{prop}
We omit the proof of Proposition \ref{prop dirichlet}, since it is an immediate consequence of Proposition~\ref{prop singular bdry}. Let us now prove our main result, Theorem \ref{thm main}.

\begin{proof}[Proof of Theorem \ref{thm main}]
We have shown in Proposition \ref{prop dirichlet} that $u=0$ on $\partial \Omega$ after the extension of boundary value. We now apply the standard argument for the maximum principle to prove $u=0$ in $\Omega$.

Suppose that there exists  $x_0\in \Omega$ such that $u(x_0)>0$. 
For any $\vep>0$ small, define a function $\phi_\vep$ by
\begin{equation}\label{penalty infinity}
\phi_\vep(x)=\vep(|x-x_0|^2+1)^{1\over 2}, \quad x\in \Oba.
\end{equation}
 Since $u$ is bounded and continuous in $\Oba$, for any $\vep>0$, there exists $\hat{x}\in \Oba$ such that
 \begin{equation}\label{u max}
 u(\hat{x}) -\phi_\vep(\hat{x})= \max_{\Oba} (u-\phi_\vep)\geq u(x_0)-\vep,
 \end{equation}
 which implies that $u(\hat{x})\geq u(x_0)$.

Since $u=0$ on $\partial \Omega$, we have $\hat{x}\in \Omega$, which by the definition of viscosity subsolutions, yields
 \begin{equation}\label{buc sub1}
 u(\hat{x}) + F(\hat{x},\nabla \phi_\vep(\hat{x}), \nabla^2\phi_\vep(\hat{x})) \le 0.
 \end{equation}
Noticing that 
\[
|\nabla \phi_\vep(\hat{x})|, |\nabla^2 \phi_\vep(\hat{x})|\to 0\quad \text{as $\vep\to 0$,}
\]
by \eqref{near zero} implied by (F4), we get
\begin{equation}\label{buc sub2}
u(x_0)\leq \limsup_{\vep\to 0}|F(\hat{x},\nabla \phi_\vep(\hat{x}), \nabla^2\phi_\vep(\hat{x}))|=0, 
\end{equation}
which is a contradiction.

Similarly, assuming $u$ attains a negative local minimum also yields a contradiction. Hence, $u\equiv 0$ in $\Omega$. 
\end{proof}
\begin{rmk}
    If $\Omega$ is additionally assumed to be bounded, then we can obtain the same results in Proposition \ref{prop dirichlet} and Theorem \ref{thm main} with \eqref{near zero} in (F4) replaced by the condition that $F(x, 0, 0)=0$ for all $x\in \Omega$. Notice that \eqref{near zero} is only used to prove \eqref{buc sub2}. In a bounded domain $\Omega$, if $u\in C(\Oba)$ attains  at some $\hat{x}\in \Omega$ a  maximum $u(\hat{x})>0$, then any constant can be utilized as a test function in place of $\phi_\vep$ to get from \eqref{buc sub1} a contradiction that $u(\hat{x})\leq 0$. A similar  argument works also for Theorem \ref{thm main} for a bounded domain. 
\end{rmk}

\subsection{Nonexistence of nontrivial bounded solutions in $C(\Omega)$}
Besides the boundedness of $u$ in $\Omega$, the constraint that $u$ can be continuously extended to $\partial \Omega$ plays a key role in our proof of Theorem \ref{thm main}. In this subsection, we attempt to drop this condition and establish the Liouville-type result for continuous solutions in a bounded $C^2$ domain.  
Our main result for this part is Theorem \ref{thm main2}, which follows from Theorem \ref{thm main} and the following regularity result.  

\begin{prop}[Induced H\"older continuity]\label{prop uc}
Let $\Omega\subset \R^n$ be a bounded domain of $C^2$ class. Assume that (F1), (F2) hold. Let $u\in C(\Omega)$ be a bounded viscosity solution to \eqref{nonlinear eq}. 
Assume that (F5) holds and let $a_F>0$ and $\alpha \in (0,1]$ be the constants appearing in (F5) corresponding to $M=\sup_{x\in \Omega}|u(x)|$. 
Then $u$ satisfies 
\begin{equation}\label{uc-omega}
|u(x)-u(y)|\leq a_F|x-y|^\alpha \quad \text{ for all $x, y\in \Omega$}.
\end{equation}
\end{prop}
\begin{proof}
Hereafter we denote by $\tilde{d}: \Oba\to [0, \infty)$ a $C^2$ extension of $d$ in $\Oba$; namely, we assume that $\tilde{d}\in C^2(\Oba)$ and $\tilde{d}=d$ in $\Oba\setminus \Omega_\delta$ for some $\delta>0$, where $\Omega_\delta:=\{x\in \Omega: d(x)\ge \delta\}$. We may also assume that $\tilde{d}\geq \tilde{C}d$ in $\Oba$ for some $\tilde{C}>0$. In light of the $C^2$ regularity of $\Omega$, such an extension can be obtained by appropriately regularizing $d$ in the region away from $\partial \Omega$.

Assume by contradiction that 
 \eqref{uc-omega} fails to hold. Then there exist $x_0, y_0\in \Omega$ such that 
\[
m:=u(x_0)-u(y_0)-a_F|x_0-y_0|^\alpha>0.
\] 
For sufficiently small $\vep>0$ depending on $x_0$, $y_0$, $m$ and $D$, it follows that
\begin{equation}\label{uc eq1}
\sup_{(x, y)\in \Omega\times \Omega}  \Psi_\vep(x, y)\geq \Psi_\vep(x_0, y_0)\geq {m\over 2},
\end{equation}
where 
\begin{equation}\label{uc penalty}
\Psi_\vep(x, y)=u(x)-u(y)-a_F|x-y|^\alpha+\vep\log {\tilde{d}(x)\over D}+\vep \log {\tilde{d}(y)\over D}.
\end{equation}
Here we take ${D}=\max_{\Oba} \tilde{d}$.
Thanks to the penalization near $\partial \Omega$, the supremum of $\Psi_\vep$ can only be attained in $\Omega\times \Omega$. Assume that $(x_\vep, y_\vep)\in \Omega\times \Omega$ is a maximizer. It is clear that  $x_\vep\neq y_\vep$ holds for each $\vep>0$ small, for otherwise $\Psi_\vep(x_\vep, y_\vep)\leq 0$, which is  a contradiction to \eqref{uc eq1}.

Let
\[
w(x, y)= a_F|x-y|^\alpha-\vep\log {\tilde{d}(x)\over D}-\vep \log {\tilde{d}(y)\over D}, \quad x, y\in \Omega. 
\]
We next apply the Crandall-Ishii Lemma \cite{CIL}. For any $\sigma>0$, there exist $X_\vep, Y_\vep\in \S^n$ such that
\[
(p_\vep, X_\vep)\in \ol{J}^{2, +} u(x_\vep), \quad (q_\vep, Y_\vep)\in \ol{J}^{2, -}u(y_\vep)
\]
with
\begin{equation}\label{uc eq2}
p_\vep=\nabla_x w(x_\vep, y_\vep), \quad q_\vep=-\nabla_y w(x_\vep, y_\vep)
\end{equation}
and 
\begin{equation}\label{uc eq3}
\begin{pmatrix}
X_\vep & 0\\
0 & -Y_\vep
\end{pmatrix}
\leq \nabla^2 w(x_\vep, y_\vep)+\sigma (\nabla^2 w(x_\vep, y_\vep))^2.
\end{equation}
Let us take $\sigma>0$ so small that $\sigma \|(\nabla^2 w(x_\vep, y_\vep))^2\|\leq \vep^2$. Setting $\eta_\vep=x_\vep-y_\vep$, by direct computations we see that
\begin{equation}\label{uc eq4}
\begin{aligned}
&\nabla_x w(x_\vep, y_\vep)=\alpha a_F |\eta_\vep|^{\alpha-2}\eta_\vep-\vep{\nabla \tilde{d}(x_\vep)\over \tilde{d}(x_\vep)},\\
&\nabla_y w(x_\vep, y_\vep)=-\alpha a_F |\eta_\vep|^{\alpha-2}\eta_\vep-\vep{\nabla \tilde{d}(y_\vep)\over \tilde{d}(y_\vep)},
\end{aligned}
\end{equation}
and 
\begin{equation}\label{uc eq5}
\begin{aligned}
&\nabla_x^2 w(x_\vep, y_\vep)=A_\vep-\vep {\nabla^2 \tilde{d}(x_\vep)\over \tilde{d}(x_\vep)}+\vep {\nabla \tilde{d}(x_\vep)\otimes \nabla \tilde{d}(x_\vep)\over \tilde{d}(x_\vep)^2}, \\
&\nabla_y^2 w(x_\vep, y_\vep)=A_\vep-\vep {\nabla^2 \tilde{d}(y_\vep)\over \tilde{d}(y_\vep)}+\vep {\nabla \tilde{d}(y_\vep)\otimes \nabla \tilde{d}(y_\vep)\over \tilde{d}(y_\vep)^2}, \\
& \nabla^2_{xy} w(x_\vep, y_\vep)=\nabla^2_{yx} w(x_\vep, y_\vep) =-A_\vep,
\end{aligned}
\end{equation}
where 
\begin{equation}\label{uc eq6}
A_\vep= \alpha(\alpha-2)a_F|\eta_\vep|^{\alpha-4} \eta_\vep\otimes \eta_\vep 
+\alpha a_F |\eta_\vep|^{\alpha-2} I_n.
\end{equation}
Recall that $\tilde{d}\in C^2(\Oba)$ and $\tilde{d}\geq \tilde{C}d$ holds in $\Oba$ for $\tilde{C}>0$. Then \eqref{uc eq2} and \eqref{uc eq4} yield
\begin{equation}\label{eq a5-1}
\left|p_\vep-\alpha a_F |\eta_\vep|^{\alpha-2}\eta_\vep\right|\leq {C_1\vep\over d(x_\vep)}, 
\qquad \left|q_\vep-\alpha a_F |\eta_\vep|^{\alpha-2}\eta_\vep\right|\leq {C_1\vep\over d(y_\vep)}
\end{equation}
for some $C_1>0$ depending on $\tilde{C}$ and a bound for $\nabla \tilde{d}$. Moreover, by \eqref{uc eq3} and \eqref{uc eq5} we get
\begin{equation}\label{eq a5-2}
\begin{pmatrix}
X_\vep & 0\\
0 & -Y_\vep
\end{pmatrix}\leq \begin{pmatrix}
A_\vep & -A_\vep\\
-A_\vep & A_\vep
\end{pmatrix}+ C_2\vep \begin{pmatrix} d(x_\vep)^{-2} I_n & 0\\ 0 & {d(y_\vep)^{-2}} I_n \end{pmatrix}+ \vep^2 I_{2n}. 
\end{equation}
for some $C_2>0$ depending on $\tilde{C}$ and a bound for both $\nabla \tilde{d}$ and $\nabla^2 \tilde{d}$.
In addition, by \eqref{uc eq6}, we have 
\[
\la A_\vep \xi, \xi\ra\leq \alpha a_F |\eta_\vep|^{\alpha-2} |\xi|^2
\]
for all $\xi\in \R^n$, which amounts to saying that 
\begin{equation}\label{eq a5-3}
A_\vep \leq \alpha a_F |\eta_\vep|^{\alpha-2}I_n.
\end{equation} 

In addition, utilizing the definition of viscosity solutions in Definition~\ref{def viscosity} and Remark~\ref{rmk jets}, we have 
\begin{equation}\label{uc-sub}
u(x_\vep)+F(x_\vep, p_\vep, X_\vep)\leq 0,
\end{equation}
\begin{equation}\label{uc-super}
u(y_\vep)+F(y_\vep, q_\vep, Y_\vep)\geq 0,
\end{equation}
which yield \eqref{uc-bounds} due to the bound $M=\sup_\Omega |u|$.

Note that the estimates \eqref{eq a5-1}, \eqref{eq a5-2} and \eqref{eq a5-3} respectively correspond to the conditions \eqref{a5-1}, \eqref{a5-2} and \eqref{uc-mod-rev} in (F5). Let $C=C_1+C_2$.
We then adopt (F5) with $\delta=C\vep$ and $\vep>0$ sufficiently small to deduce that
\begin{equation}\label{uc eq7}
F(y_\vep, q_\vep, Y_\vep)-F(x_\vep, p_\vep, X_\vep)\leq \beta a_F|x_\vep-y_\vep|^\alpha+CL\vep,
\end{equation}
where $\beta \in (0,1)$ and $L>0$ are the constants in (F5). 
From \eqref{uc-sub} and \eqref{uc-super} again, we are thus led to
\[
u(x_\vep)-u(y_\vep)\leq F(y_\vep, q_\vep, Y_\vep)-F(x_\vep, p_\vep, X_\vep).
\]
It follows immediately from \eqref{uc eq7} that 
\[
u(x_\vep)-u(y_\vep)\leq \beta a_F|x_\vep-y_\vep|^\alpha+CL\vep.
\]
In view of \eqref{uc eq1}, we then get 
\begin{equation}\label{uc eq8}
a_F|x_\vep-y_\vep|^\alpha-\vep\log {\tilde{d}(x_\vep)\over D}-\vep \log {\tilde{d}(y_\vep)\over D} \leq \beta a_F|x_\vep-y_\vep|^\alpha+CL\vep
\end{equation}

Since $\tilde{d}(x_\vep), \tilde{d}(y_\vep)\leq D$, we have 
\[
 a_F|x_\vep-y_\vep|^\alpha\leq \beta a_F|x_\vep-y_\vep|^\alpha+CL\vep,
\]
which, because of the condition $\beta \in(0,1)$, yields $|x_\vep-y_\vep|\to 0$ as $\vep\to 0$.

We thus can take a subsequence such that $x_\vep, y_\vep \to \hat{x}$ for some $\hat{x}\in \Oba$. If $\hat{x}\in \Omega$, then by the continuity of $u$, we have  $\Psi_\vep(x_\vep, y_\vep)\to 0$ as $\vep\to 0$, which contradicts \eqref{uc eq1}. It remains to discuss the case that $\hat{x}\in \partial \Omega$. In this case, since 
\[
\tilde{d}(x_\vep)=d(x_\vep)\to 0, \quad \tilde{d}(y_\vep)=d(y_\vep)\to 0, 
\]
by \eqref{uc eq8}, it is easily seen that when $\vep>0$ is sufficiently small,
\[
a_F|x_\vep-y_\vep|^\alpha\leq \beta a_F|x_\vep-y_\vep|^\alpha,
\]
which is a contradiction to the condition $\beta \in(0,1)$ as well as the fact that $x_\vep\neq y_\vep$.
\end{proof}

\section{Generalization for parabolic equations}\label{sec:general}

Our results in the preceding sections can be further generalized for a class of fully nonlinear parabolic equations \eqref{parabolic eq} with a given $T>0$, where $H \in C(\Omega\times \R\times \R^n\times \S^n)$ is an elliptic operator, that is, 
\[
H(x, r, p, X)\leq H(x, r, p, Y) 
\]
for all $x\in \Omega$, $r \in \R$, $p\in \R^n$ and $X,Y\in \S^n$ satisfying $X\geq Y$.

One can define viscosity solutions in this parabolic case in a similar manner as in Definition \ref{def viscosity} and Remark \ref{rmk jets}. We give the precise definition below for the reader's convenience. For notation convenience, hereafter we denote $\Omega_T=\Omega\times (0, T)$.

\begin{defi}\label{def parabolic viscosity}
A locally bounded function $u\in USC(\Omega_T)$ (resp., $LSC(\Omega_T)$) is called a viscosity subsolution (resp., supersolution) of \eqref{parabolic eq} if whenever there exist $(x_0, t_0)\in \Omega_T$ and $\varphi\in C^2(\Omega_T)$ such that $u-\varphi$ attains a local maximum (resp., minimum) at $(x_0, t_0)$, we have 
\[
\varphi_t (x_0, t_0)+H(x_0, u(x_0, t_0), \nabla \varphi(x_0, t_0), \nabla^2 \varphi(x_0, t_0)) \leq 0
\]
\[
\left(\text{resp.,}\ \ \varphi_t (x_0, t_0)+H(x_0, u(x_0, t_0), \nabla \varphi(x_0, t_0), \nabla^2 \varphi(x_0, t_0))\geq 0 \right).
\]
A function $u\in C(\Omega_T)$ is called a viscosity solution of \eqref{parabolic eq} if it is both a viscosity subsolution and a viscosity supersolution of \eqref{parabolic eq}.
\end{defi}

Similarly to the elliptic case pointed out in Remark \ref{rmk jets}, we can also use parabolic semijets to define viscosity solutions of \eqref{parabolic eq}. For every $(x,t)\in \Omega_T$, the parabolic semijet $P^{2, +} u(x, t)$ is defined as the set of all $(a,p,X) \in \R \times \R^n\times \S^n$ satisfying
\[
u(y,s)\leq u(x,t)+a(s-t) +\la p, y-x \ra+ {1\over 2}\la X (y-x), y-x\ra+o(|y-x|^2+|s-t|)
\]
for all $(y,s)$ near $(x,t)$, and $P^{2, -}u(x, t)$ is defined as $-P^{2, +} (-u)(x, t)$. An equivalent definition of viscosity solutions of \eqref{parabolic eq} using the semijets is as follows. 

A function $u\in USC(\Omega_T)$ (resp., $u\in LSC(\Omega_T)$) is a viscosity subsolution (resp., supersolution) of \eqref{parabolic eq} if and only if 
\[
a+H(x_0, u(x_0, t_0), p_0, X_0)\leq 0 \quad \text{(resp., $\geq 0$)}
\]
holds for all $(a_0, p_0, X_0)\in P^{2, +} u(x_0,t_0)$ (resp., $(a_0, p_0, X_0)\in P^{2, -}u(x_0,t_0)$). One can replace $P^{2, \pm} u$ by $\ol{P}^{2, \pm} u$ if $H$ is assumed to be continuous. We refer to \cite{CIL} for more details about these notions and standard results.

Let us provide a parabolic version of Proposition \ref{prop singular bdry}.

\begin{prop}[Implicit boundary condition]\label{prop general bdry}
Let $\Omega\subset \R^n$ be a domain satisfying (A). Let $H\in C(\Omega\times \R\times \R^n\times \S^n)$ be an elliptic operator. Let $T>0$, $x_0\in \partial \Omega$ and $0<t_0<T$. Assume that there exists $G_0\in C(\R)$ depending on $x_0$ such that 
\begin{equation}\label{degeneracy general}
\begin{aligned}
\sup \Big\{|H(x, r, p, X)-G_{0}(r)|:& \  x\in \Omega,  \ |x-x_0|\leq \delta, \ r\in \R,  \\
&  |p|\leq \omega(\delta)/\delta, |X|\leq \omega(\delta)/\delta^2 \Big\}\to  0\quad \text{as $\delta\to 0$}.
\end{aligned}
\end{equation}
If $u\in USC(\Omega_T)$ (resp., $LSC(\Omega_T)$) is a bounded subsolution (resp., supersolution) of \eqref{parabolic eq} 
and $u^\ast$ (resp., $u_\ast$) is continuous at $(x_0, t_0)\in \partial\Omega\times (0, T)$, then $u^\ast\in USC(\Oba\times (0, T))$ (resp., $u_\ast\in LSC(\Oba\times (0, T))$ satisfies 
\[
\begin{aligned}
u_t(x_0, t_0)+G_{0}(u^\ast(x_0, t_0))\leq 0\quad
 \left(\text{resp., } u_t(x_0, t_0)+G_{0}(u_\ast(x_0, t_0))\geq 0 \right)
\end{aligned}
\]
in the sense that 
\[
\begin{aligned}
\varphi_t(x_0, t_0)+G_{0}(u^\ast(x_0, t_0))\leq 0\quad
 \left(\text{resp., } \varphi_t(x_0, t_0)+G_{0}(u_\ast(x_0, t_0))\geq 0 \right)
\end{aligned}
\]
holds whenever there exists $\varphi\in C^2(\Oba\times (0, T))$ such that $u^\ast-\varphi$ (resp., $u_\ast-\varphi$) attains a local maximum (resp., local minimum) in $\Oba_T$ at $(x_0, t_0)$. 
\end{prop}
\begin{proof}
The proof is similar that of Proposition \ref{prop singular bdry}, but our argument here involves test functions. For $r>0$, let $\N_r(x_0)=B_r(x_0)\cap \Omega$. 
By adding a proper quadratic function to $\varphi$, we may assume that $u^\ast-\varphi$ takes a strict local maximum at $(x_0,t_0)\in \partial \Omega\times (0, T)$. By the upper semicontinuity of $u^\ast$, for $r>0$ sufficiently small, 
\[
\Phi_\vep(x, y, t)=u^\ast(x, t)-\varphi(y, t)-\left|{x-y\over \vep}-\nu(y)\right|^4 -{|y-x_0|^2\over \vep^2}
\]
attains its maximum in $\ol{\N_r(x_0)}\times \ol{\N_r(x_0)}\times [t_0-r, t_0+r]$ at  a point $(x_\vep, y_\vep, t_\vep)\in \ol{\N_r(x_0)}\times \ol{\N_r(x_0)}\times [t_0-r, t_0+r]$. It follows that 
\[
\Phi_\vep(x_\vep, y_\vep, t_\vep) \geq \Phi_\vep(x_0+\vep \nu(x_0), x_0, t_0),
\]
which yields 
\begin{equation}\label{dyn-bdry eq1}
\left|{x_\vep-y_\vep\over \vep}-\nu(y_\vep)\right|^4+{|y_\vep-x_0|^2\over \vep^2}\leq u^\ast(x_\vep, t_\vep)-\varphi(y_\vep, t_\vep)-u^\ast(x_0+\vep\nu(x_0), t_0) +\varphi(x_0, t_0). 
\end{equation}
Then there exists $C>0$ such that 
\[
|y_\vep-x_0|\leq C\vep,\quad |x_\vep-x_0|\leq C\vep. 
\]
 By taking a subsequence, still indexed by $\vep>0$ for simplicity, we have $x_\vep, y_\vep\to x_0$ and $t_\vep\to \hat{t}$ as $\vep\to 0$ for some $\hat{t}\in [t_0-r, t_0+r]$. By the assumptions that $\varphi\in C^2(\Oba\times (0, T))$,  $u\in USC(\Oba\times (0, T))$ and $u^\ast$ is continuous at $(x_0, t_0)$, we send  $\vep\to 0$ in \eqref{dyn-bdry eq1} to get
\[
u^\ast(x_0, \hat{t})-\varphi(x_0, \hat{t})-u^\ast(x_0, t_0)+\varphi(x_0, t_0)\geq 0. 
\]
Since $(x_0, t_0)$ is the only maximizer of $u^\ast-\varphi$ in $\Oba\times (0, T)$ near $(x_0, t_0)$, we have $\hat{t}=t_0$. It then follows that $(x_\vep, t_\vep)\to (x_0, t_0)$ and $(y_\vep, t_\vep)\to (x_0, t_0)$  as $\vep\to 0$. 

In view of \eqref{dyn-bdry eq1}, we also have 
\[
\left|{x_\vep-y_\vep\over \vep}-\nu(y_\vep)\right|\to 0\quad \text{as $\vep\to 0$,}
\]
which yields $x_\vep-(y_\vep+\vep \nu(y_\vep))=o(\vep)$ for $\vep>0$ sufficiently small. We can use the assumption (A) to deduce that $x_\vep\in \Omega$ for all $\vep>0$ small. Applying the definition of viscosity subsolutions, we obtain 
\[
(\psi_\vep)_{t}(x_\vep, t_\vep)+H(x_\vep, u(x_\vep, t_\vep), \nabla \psi_\vep(x_\vep, t_\vep), \nabla^2 \psi_\vep(x_\vep, t_\vep))\leq 0,
\]
where
\[
\psi_\vep(x, t):= \varphi(y_\vep, t)+\left|{x-y_\vep\over \vep}-\nu(y_\vep)\right|^4+{|y_\vep-x_0|^2\over \vep^2}.
\]
By direct computations, we have 
\[
(\psi_\vep)_t(x_\vep, t_\vep)=\varphi_t(y_\vep, t_\vep) \to \varphi_t(x_0, t_0), 
\]
\[
\vep |\nabla \psi_\vep(x_\vep, t_\vep)|\to 0, \quad  \vep^2 |\nabla^2 \psi_\vep(x_\vep, t_\vep)|\to 0\quad \text{ as $\vep\to 0$.}
\]
Letting $\vep\to 0$ with an application of \eqref{degeneracy general}, we deduce that 
\[
\varphi_t(x_0, t_0)+ G_{0}(u^\ast(x_0, t_0))\leq 0.
\]
The statement for supersolutions can be proved in a symmetric argument.  
\end{proof}

In general, the result of Proposition \ref{prop general bdry} holds only for $(x_0, t_0)\in \partial \Omega\times (0, T)$. If we consider $t_0=0$ or $t_0=T$ for $u\in USC(\overline{\Omega_T})$ and $\varphi\in C^2(\overline{\Omega_T})$ in our proof above, then it may happen that $t_\vep=0$ or $t_\vep=T$ for the maximizer $(x_\vep, y_\vep, t_\vep)$ of $\Phi_\vep$. In this case, we are not able to derive any condition at $(x_0, t_0)$ from the equation in $\Omega_T$.

An immediate consequence of Proposition \ref{prop general bdry} is as follows. 

\begin{cor}[Dynamic boundary condition]\label{cor dyn}
    Let $\Omega\subset \R^n$ be a domain satisfying (A). Let $H\in C(\Omega\times \R\times \R^n\times \S^n)$ be an elliptic operator. Assume that there exists $G\in C(\partial\Omega\times \R)$ such that \eqref{degeneracy general2} holds. 
Let $u\in C(\Omega_T)$ be a solution of \eqref{parabolic eq} that can be extended continuously to $\partial \Omega\times (0, T)$.  Then $u\in C(\Oba\times (0, T))$ satisfies 
\begin{equation}\label{dyn cond}
u_t+G(x, u)=0 \quad \text{on $\partial \Omega\times (0, T)$}    
\end{equation}
in the following sense.
\begin{enumerate}
    \item[1)] Whenever there exist $\varphi\in C^2(\Oba\times (0, T))$ and $(x_0, t_0)\in \partial \Omega\times (0, T)$ such that $u-\varphi$ attains a local maximum in $\Oba\times (0, T)$ at $(x_0, t_0)$, there holds
    \[
    \varphi_t(x_0, t_0)+G(x_0, u(x_0, t_0))\leq 0.  
    \]
    \item[2)] Whenever there exist $\varphi\in C^2(\Oba\times (0, T))$ and $(x_0, t_0)\in \partial \Omega\times (0, T)$ such that $u-\varphi$ attains a local minimum in $\Oba\times (0, T)$ at $(x_0, t_0)$, there holds
    \[
    \varphi_t(x_0, t_0)+G(x_0, u(x_0, t_0))\geq 0.  
    \]
\end{enumerate}
\end{cor}

To prove this, we simply apply Proposition \ref{prop general bdry} with $G_0=G(x_0, \cdot\ )$ for the local maximizer or minimizer $(x_0, t_0)\in \partial \Omega\times (0, T)$ of $u-\varphi$, noting that \eqref{degeneracy general2} implies \eqref{degeneracy general}. The details are omitted here.

Corollary \ref{cor dyn} states that, under a strong uniform convergence  of $H(x, r, p, X)$ to $G(x, r)$ as $d(x)\to 0$,  $u$ actually satisfies the dynamic boundary condition \eqref{dyn cond} 
in the viscosity sense, which again arises from the degeneracy of the operator rather than being imposed explicitly as in the classical viscosity solution theory.

  In view of Proposition~\ref{prop general bdry}, we can generalize the results of Proposition~\ref{prop dirichlet} and Theorem~\ref{thm main} to a more general class of equations. 

\begin{thm}[Parabolic Liouville-type theorem for bounded solutions in $C(\Oba\times [0, T))$]\label{thm liouville general}
Let $\Omega\subset \R^n$ be a domain satisfying (A). Let $H\in C(\Omega\times \R\times \R^n\times \S^n)$ be an elliptic operator satisfying \eqref{degeneracy general2} for $G\in C(\partial\Omega\times \R)$. Assume that $r\mapsto H(x, r, p, X)$ is nondecreasing for all $x\in \Omega$, $p\in \R^n, X\in \S^n$.  
Assume in addition that 
\begin{equation}\label{general op0}
    \sup_{x\in \Omega} |H(x, 0, p, X)|\to 0  \quad\text{as $|p|, |X|\to 0$.}    
\end{equation}
If $u\in C(\Omega\times [0, T))$ is a bounded viscosity solution to \eqref{parabolic eq} that can be extended to a function in $C(\Oba\times [0, T))$ with $u(\cdot, 0)=0$ in $\Oba$, then $u\equiv 0$ in $\Omega\times [0, T)$. \end{thm}

\begin{proof}
 The proof, analogous to that of Theorem~\ref{thm main}, follows the standard arguments for maximum principle. Suppose that $u\in C(\Oba\times [0, T))$ takes a positive value at $(x_0, t_0)$. Then for any small $\sigma>0$ fixed, 
\begin{equation}\label{time penalty}
(x, t)\mapsto u(x, t)-\phi_\vep(x)-{\sigma \over T-t}
\end{equation}
attains a positive maximum at $(x_\vep, t_\vep)\in \Oba\times (0, T)$ for any $\vep>0$ small, where $\phi_\vep$ is taken as in \eqref{penalty infinity}. Indeed, we have 
\[
u(x_0, t_0)-\phi_\vep(x_0)-{\sigma\over T-t_0}>0
\]
for sufficiently small $\vep, \sigma>0$ depending on $x_0$ and $t_0$. Fix this $\sigma$. Since $\phi_\vep\geq 0$, $u$ is bounded in $\Omega_T$, and $u(\cdot, 0)=0$ in $\Oba$, we see that 
\[
\sup_{x\in \Oba} \left(u(x, t)-\phi_\vep(x)-{\sigma \over T-t}\right)\leq 0 
\]
when $t=0$ or $t\geq T-\sigma/\|u\|_{\infty}$. In addition, we have $\phi_\vep(x)\to \infty$ as $|x|\to \infty$. Hence, the function as in \eqref{time penalty} has a maximizer $(x_\vep, t_\vep)\in \Oba\times (0, T)$.

We can further show that $x_\vep\notin \partial \Omega$, for otherwise we use Proposition \ref{prop general bdry} to get 
\[
{\sigma\over (T-t_\vep)^2}+G(x_\vep, u(x_\vep, t_\vep))\leq 0, 
\]
which is a contradiction, since 
\[
G(x_\vep, u(x_\vep, t_\vep))=\lim_{\Omega\ni y\to x_\vep} H(y, u(x_\vep, t_\vep), 0, 0)\geq \lim_{\Omega\ni y\to x_\vep} H(y, 0, 0, 0)=0
\]
holds due to \eqref{degeneracy general2} and the monotonicity of $r\mapsto H(x, r, p, X)$.
For $\vep>0$ small, the maximum cannot occur at $x_\vep\in \Omega$ either. Notice that an interior maximum yields
\[
{\sigma \over (T-t_\vep)^2}+H(x_\vep, u(x_\vep, t_\vep), \nabla \phi_\vep(x_\vep), \nabla^2 \phi_\vep(x_\vep))\leq 0,
\]
which, by \eqref{general op0} and the monotonicity of $r\mapsto H(x, r, p, X)$,  implies a contradiction again if $\vep> 0$ is taken small. One can similarly show that $u$ does not take negative values. Hence, $u\equiv 0$ in $\Omega_T$.
\end{proof}

It is also possible to extend the arguments for Proposition~\ref{prop uc} and Theorem \ref{thm main2} to parabolic equations \eqref{parabolic eq} when 
\[
H(x, r, p, X)=r+F(x, p, X), \quad (x, r, p, X)\in \Omega\times \R\times \R^n\times \S^n.
\]
To this end, we need to strengthen (F5) by asking constants $a_F, L, \alpha, \beta$ to be independent of $M>0$ and impose the following additional assumption. 
\begin{enumerate}
\item[(F6)] There exist $b_F>0$ and $\ell>0$ such that
\begin{equation}\label{parabolic uc-cond}
F(y, q, Y) -  F(x, p, X) \leq  b_F c |x-y|^2+\ell\delta
\end{equation}
holds for any $c,\delta>0$, $x, y\in \Omega$, $p, q\in \R^n$ and $X, Y\in \S^n$ satisfying
\begin{equation}\label{parabolic1}
\left|p -c (x-y)\right| \leq {\delta\over d(x)},  
\qquad
\left|q -c (x-y)\right| \leq {\delta\over d(y)}, 
\end{equation}
and 
\begin{equation}\label{parabolic2}
 \begin{pmatrix} X & 0\\ 0& -Y\end{pmatrix}\leq c\begin{pmatrix} I_n & -I_n \\ -I_n & I_n\end{pmatrix}+\delta \begin{pmatrix} {d(x)^{-2}} I_n & 0\\ 0 & {d(y)^{-2}} I_n \end{pmatrix}+ \delta I_{2n}.
 \end{equation}
\end{enumerate}
Heuristically speaking, we shall use the strengthened (F5) to obtain the H\"older regularity of solutions in space as in the elliptic case while (F6) is for the regularity in time. 

We present parabolic variant of Theorem~\ref{thm main2} as follows. 
 \begin{thm}[Parabolic Liouville-type theorem for bounded solutions in $C(\Omega\times [0, T))$]\label{thm main3}
Let $\Omega\subset \R^n$ be a bounded domain of $C^2$ class. Assume that $F$ satisfies (F1)--(F4), (F6) and (F5) with $a_F, L>0$, $\alpha \in (0,1]$ and $\beta \in (0,1)$ independent of $M>0$.
If $u\in C(\Omega\times [0, T))$ is a bounded viscosity solution to \eqref{parabolic eq} satisfying initial condition $u(\cdot, 0)=0$ in $\Omega$ and there exists a modulus of continuity $\omega_0: [0, \infty)\to [0, \infty)$, strictly increasing and smooth in $(0, \infty)$ such that $\sup_{x\in \Omega} |u(x,t)| \le \omega_0 (t)$ for $t\in [0,T)$, then $u\equiv 0$ in $\Omega_T$. 
\end{thm}

\begin{proof}
The proof is divided into two steps. In the first step, we show that
\begin{equation}\label{time regularity}
u(x,t)-u(x,s) \le \omega_0 (|t-s|)
\end{equation}
for all $s, t\in[0,T)$ and $x \in \Omega$.
This inequality \eqref{time regularity} holds when $t=0$ or $s=0$ because of the assumption $\sup_{x\in \Omega} |u(x,t)| \le \omega_0 (t)$ and the homogeneous initial condition.
Assume that there exist $x_0\in \Omega$ and $t_0,s_0\in (0,T)$ such that
\begin{equation*}
\mu:= u(x_0,t_0) - u(x_0,s_0) -\omega_0(|t_0-s_0|) >0.
\end{equation*}
It follows that for $\vep>0$ small, 
\begin{equation}\label{positivity1}
\sup_{\Omega \times [0,T) \times [0,T)} \Phi_\vep(x, t, s)\geq {\mu \over 2}>0,
\end{equation}
where 
\begin{align*}
\Phi_\vep(x, t, s)=u(x, t)-u(x, s)-\omega_0(|t-s|)+2\vep\log {\tilde{d}(x)\over D}-{\vep \over T-t}-{\vep \over T-s}.
\end{align*}
Let $(x_\vep, t_\vep, s_\vep)$ be a maximizer of $\Phi_\vep$ on $\Omega\times [0,T) \times [0,T)$.
Here, for a fixed $\vep>0$ (to be determined later) we also consider
\begin{align*}
\Phi_{\vep, \delta}(x, y, t, s)=u(x, t)-u(y, s)-& \frac{|x-y|^2}{\delta}-\omega_0(|t-s|)\\
& +\vep\log {\tilde{d}(x)\over D}+ \vep\log {\tilde{d}(y)\over D}-{\vep \over T-t}-{\vep \over T-s}
\end{align*}
and denote by $(x_{\vep, \delta},\, y_{\vep, \delta},\, t_{\vep, \delta},\, s_{\vep, \delta})$ a maximizer of $\Phi_{\vep, \delta}$ on $\Omega\times \Omega \times [0,T) \times [0,T)$. Note that $t_{\vep, \delta}, s_{\vep, \delta}>0$ due to the assumption $\sup_{x\in \Omega} |u(x,t)| \le \omega_0 (t)$.

Since \eqref{positivity1} implies $\Phi_{\vep, \delta}(x_{\vep, \delta},\, y_{\vep, \delta},\, t_{\vep, \delta},\, s_{\vep, \delta})>0$ and the solution $u$ is bounded, we have
\begin{equation}\label{delta_convergence}
|x_{\vep, \delta} - y_{\vep, \delta}|\to 0 \quad\text{as $\delta \to 0$} 
\end{equation}
uniformly for $\vep>0$. We also have $t_{\vep, \delta}\neq s_{\vep, \delta}$ for any $\delta$ small. 
In fact, in view of \eqref{space regularity}, assuming $t_{\vep, \delta}= s_{\vep, \delta}$ yields 
\begin{align*}
\frac{\mu}{2} 
 \le \Phi_{\vep, \delta}(x_{\vep, \delta}, y_{\vep, \delta}, t_{\vep, \delta}, t_{\vep, \delta})\le u(x_{\vep, \delta}, t_{\vep, \delta})-u(y_{\vep, \delta}, t_{\vep, \delta}) \le a_F|x_{\vep, \delta}-y_{\vep, \delta}|^\alpha, 
\end{align*}
which is a contradiction to \eqref{delta_convergence}. 
We therefore deduce $t_{\vep, \delta}\neq s_{\vep, \delta}$ when $\delta>0$ is taken small. 

By taking a subsequence of $\delta$ for each $\vep$, we can obtain stronger convergence than \eqref{delta_convergence}. From the boundedness of  $\Omega$, for each fixed $\vep$ there exist $\hat{x}_\vep \in \Oba$, $\hat{t}_\vep, \hat{s}_\vep \in [0, T]$ and a subsequence such that 
$x_{\vep, \delta}, y_{\vep, \delta}\to \hat{x}_\vep$, $t_{\vep, \delta} \to\hat{t}_\vep$, $s_{\vep, \delta} \to \hat{s}_\vep$ as $\delta \to 0$. Since 
\begin{equation}\label{Phi_max}
\Phi_{\vep, \delta}(x_{\vep, \delta}, y_{\vep, \delta}, t_{\vep, \delta}, s_{\vep, \delta})
\ge \Phi_{\vep}(x_{\vep}, t_{\vep}, s_{\vep})>0, 
\end{equation}
we have $\hat{x}_\vep\in \Omega$ and $\hat{t}_\vep, \hat{s}_\vep \in (0, T)$. By further taking a subsequence if necessary, we may also assume that 
\[
\frac{|x_{\vep, \delta}-y_{\vep, \delta}|^2}{\delta}\to c_\vep\quad \text{as $\delta\to 0$}
\]
for some $c_\vep\geq 0$. Letting $\delta \to 0$ in \eqref{Phi_max}, we see that 
\[
\Phi_\vep(\hat{x}_\vep, \hat{t}_\vep, \hat{s}_\vep)-c_\vep\geq \Phi_\vep(x_\vep, t_\vep, s_\vep).
\]
Since $(x_\vep, t_\vep, s_\vep)$ is a maximizer of $\Phi_\vep$, it follows that $c_\vep=0$, that is, 
\begin{equation}\label{delta_convergence2}
\frac{|x_{\vep, \delta}-y_{\vep, \delta}|^2}{\delta}\to 0 \quad \text{as $\delta\to 0$. }    
\end{equation}

The rest of the proof is similar to the arguments for Proposition~\ref{prop uc}. We apply the parabolic Crandall-Ishii Lemma \cite[Theorem 8.3]{CIL} to obtain, for any $\sigma>0$ small enough, $(\tilde{h}, \tilde{p}, \tilde{X})\in \ol{P}^{2, +} u(x_{\vep, \delta}, t_{\vep, \delta})$  and $(\tilde{k}, \tilde{q}, \tilde{Y})\in \ol{P}^{2, -} u(y_{\vep, \delta}, s_{\vep, \delta})$ such that
\begin{equation}\label{parabolic-t}
\tilde{h}-\tilde{k}=
{\vep\over (T-t_{\vep, \delta})^2}+{\vep\over (T-s_{\vep, \delta})^2},
\end{equation}
\[
\tilde{p}-{2\over \delta}(x_{\vep, \delta}- y_{\vep, \delta})=-{\vep\over \tilde{d}(x_{\vep, \delta})}\nabla \tilde{d}(x_{\vep, \delta}), \quad \tilde{q}-{2\over \delta}(x_{\vep, \delta}- y_{\vep, \delta})={\vep\over \tilde{d}(y_{\vep, \delta})}\nabla \tilde{d}(y_{\vep, \delta}), 
\]
\[
 \begin{pmatrix} \tilde{X} & 0\\ 0& -\tilde{Y}\end{pmatrix}\leq {2\over \delta}\begin{pmatrix} I_n & -I_n \\ -I_n & I_n\end{pmatrix}+C\vep\begin{pmatrix} {d(x_{\vep, \delta})^{-2}} I_n & 0\\ 0 & {d(y_{\vep, \delta})^{-2}} I_n \end{pmatrix}+ \vep^2 I_{2n}
\]
for some $C>0$. The definitions of parabolic semijets $P^{2, \pm} u$ and their closures $\ol{P}^{2, \pm}u$ of a function $u$ are omitted here but can also be found in \cite{CIL}.  

Adopting (F6), we are led to
\begin{equation}\label{parabolic-x}
F(y_{\vep, \delta}, \tilde{q}, \tilde{Y})-F(x_{\vep, \delta}, \tilde{p}, \tilde{X})\leq {2b_F\over \delta} |x_{\vep, \delta}-y_{\vep, \delta}|^2+C\ell \vep
\end{equation}
for some constant $C>0$. Applying the definition of viscosity sub- and supersolutions to $u$ together with \eqref{parabolic-t}\eqref{parabolic-x}, we have
\begin{equation}\label{parabolic-contradiction}
u(x_{\vep, \delta}, t_{\vep, \delta})- u(y_{\vep, \delta}, s_{\vep, \delta}) 
\le \frac{2b_F}{\delta} |x_{\vep, \delta}-y_{\vep, \delta}|^2 +C\ell \vep
\end{equation}
for some $C>0$. Here, since
$\Phi_{\vep, \delta}(x_{\vep, \delta}, y_{\vep, \delta}, t_{\vep, \delta}, s_{\vep, \delta}) \ge \mu /2$, we see that
\[
\frac{\mu}{2}
\le \frac{2b_F}{\delta} |x_{\vep, \delta}-y_{\vep, \delta}|^2 +C\ell \vep
\]
Choosing $\vep>0$ small such that 
\[
C\ell \vep \leq {\mu\over 4}
\]
and sending $\delta \to 0$ in \eqref{parabolic-contradiction}, by \eqref{delta_convergence2} we obtain a contradiction that $\mu/2\leq \mu/4$. Hence, the inequality \eqref{time regularity} follows. 

 In the second step, we show that
\begin{equation}\label{space regularity}
u(x,t)-u(y,t) \le a_F |x-y|^\alpha
\end{equation}
for all $t\in[0,T)$ and $x,y \in \Omega$. We omit the detailed proof for this inequality, as the main part resembles the proof of Proposition~\ref{prop uc}. The only difference is that, instead of \eqref{uc penalty}, we consider the positive maximum of 
\[
\Psi_\vep(x, y, t)=u(x, t)-u(y, t)-a_F |x-y|^\alpha+\vep\log {\tilde{d}(x)\over D}+\vep \log {\tilde{d}(y)\over D}-{\vep \over T-t}
\]
for $\vep>0$ small. In this case, the maximum cannot be attained at $t=0$ due to the homogeneous initial condition. We then use the parabolic Crandall-Ishii Lemma \cite[Theorem 8.3]{CIL} again to establish viscosity inequalities and complete the proof in the same way as Proposition~\ref{prop uc}.

We conclude the proof of the uniform continuity of the solution $u$ by combining \eqref{time regularity} and \eqref{space regularity}. Hence, it can be continuously extended to $\Oba \times [0,T)$. The conclusion that $u\equiv 0$ in $\Omega_T$ is now a direct consequence of Theorem~\ref{thm liouville general}. 
\end{proof}

\section{Examples}\label{sec:example}
Our general results in Theorem \ref{thm main} and Theorem \ref{thm liouville general} apply to various degenerate elliptic or parabolic operators. The assumption of uniform continuity of solutions is clarified in Theorem \ref{thm main2} and 
Theorem \ref{thm main3} under more specific conditions. 
In this section, we present concrete equations that help understand those assumptions.  

Let us provide a sufficient condition for (F5) for our convenience later. 

\begin{lem}[Sufficient condition on degeneracy]\label{lem f5}
Let $\Omega$ be a bounded $C^2$ domain in $\R^n$. Let $\sigma: \Omega\to \R^{n\times m}$ and $\psi: \Omega\to \R^n$ be continuous. Assume that there exists $L_1>0$ such that 
\begin{equation}\label{decompose1}
|\sigma(x)-\sigma(y)|\leq L_1|x-y|, \quad |\psi(x)-\psi(y)|\leq L_1|x-y|,
\end{equation}
and 
\begin{equation}\label{decompose2}
\sup_{x\in \Omega}{1\over d(x)}\max\{|\sigma(x)|, |\psi(x)|\}\leq L_1. 
\end{equation}
Assume that there exists $L_2>0$ such that $F\in C(\Omega\times \R^n\times \S^n)$ satisfies 
\begin{equation}\label{decompose3}
\begin{aligned}
F(y, q, Y)- F(x, p, X)\leq & L_2 \lambda_{max} \left(\sigma(x)^TX \sigma(x) - \sigma(y)^T Y \sigma(y)\right)\\
&\quad +L_2|\la \psi(x),  p\ra-\la \psi(y), q\ra|        
\end{aligned}
\end{equation}
for all $x, y\in\Omega$, $p, q\in \R^n$ and $X, Y\in \S^n$, where $\lambda_{max}(A)$ denotes the largest eigenvalue of $A\in \S^m$. Then (F5) holds for some $a_F,L>0$, $\alpha\in (0, 1]$ and $\beta\in (0, 1)$ all independent of $M>0$. 
\end{lem}
\begin{proof}
Fix $M>0$ arbitrarily. Suppose that we have $\delta>0$, $x, y\in \Omega$ with $x\neq y$, $p, q\in \R^n$ and $X, Y\in \S^n$ satisfying \eqref{uc-bounds}, and 
\begin{equation}\label{uc-cond-gradient}
\begin{aligned}
&|p-\alpha a_F |x-y|^{\alpha-2}(x-y)| \leq \delta d(x)^{-1}, \\
&|q- \alpha a_F |x-y|^{\alpha-2}(x-y)|\leq  \delta d(y)^{-1},  
\end{aligned}
\end{equation}
as well as
\begin{equation}\label{matrix inequality}
 \begin{pmatrix} X & 0\\ 0& -Y\end{pmatrix}\leq \begin{pmatrix} A & -A \\ -A & A\end{pmatrix}+\delta \begin{pmatrix} {d(x)^{-2}} I_n & 0\\ 0 & {d(y)^{-2}} I_n \end{pmatrix}+ \delta I_{2n} 
\end{equation}
with $A=A(x, y)\in \S^n$ satisfying 
\begin{equation}\label{uc-mod-holder}
A\leq \alpha a_F |x-y|^{\alpha-2} I_n. 
\end{equation}
It follows from \eqref{decompose2} and \eqref{uc-cond-gradient} that 
\[
\begin{aligned}
&|\la p, \psi(x)\ra-\alpha a_F |x-y|^{\alpha-2}\la x-y, \psi(x)\ra| \leq L_1\delta, \\
&|\la q, \psi(y)\ra - \alpha a_F |x-y|^{\alpha-2}\la x-y, \psi(y)\ra|\leq L_1 \delta,  
\end{aligned}
\]
which by the Lipschitz continuity of $\psi$ in \eqref{decompose1} implies that 
\begin{equation}\label{lip1}
|\la p, \psi(x)\ra-\la q, \psi(y)\ra|\leq \alpha a_F L_1 |x-y|^\alpha +2L_1\delta. 
\end{equation}
On the other hand, multiplying \eqref{matrix inequality} by $(\sigma(x)\xi, \sigma(y)\xi)$ from left and right for any $\xi\in \R^m$, we get
\[
\begin{aligned}
&\la X\sigma(x)\xi,\sigma(x)\xi \ra- \la  Y\sigma(y)\xi, \sigma(y)\xi \ra \\
&\le \la A(\sigma(x)-\sigma(y))\xi,(\sigma(x)-\sigma(y))\xi \ra 
+ 2\delta|\xi|^2\left\{\sup_{x\in \Omega}{|\sigma(x)|^2\over d(x)^2}+ \sup_{x\in \Omega}{|\sigma(x)|^2}\right\}, 
\end{aligned}
\]
which by \eqref{decompose1}, \eqref{decompose2}, \eqref{uc-mod-holder} and the boundedness of $\Omega$ yields 
\begin{equation}\label{lip2}
\sigma(x)^TX \sigma(x)- \sigma(y)^T Y \sigma(y    )\leq (\alpha a_F L_1^2 |x-y|^\alpha+2\delta L_1^2(1+D^2)) I_n,
\end{equation}
where $D=\max_{\Oba} d$. 
Applying \eqref{decompose3} with \eqref{lip1} and \eqref{lip2}, we obtain
\[
F(y, q, Y)-F(x, p, X)\leq \alpha a_F (1+L_1)L_1L_2|x-y|^\alpha+2\delta L_1L_2(1+L_1(1+D^2)).
\]
Letting $\alpha>0$ small so that $\beta=\alpha (1+L_1)L_1L_2<1,$
we complete the verification of (F5) with  $L=2L_1L_2(1+L_1(1+D^2))$. It is clear that none of $a_F, L$, $\alpha, \beta$ depends on $M>0$.
\end{proof}

Now let us discuss several concrete examples. 
\begin{example}[Laplace-type equation]\label{ex1}
Let $\Omega \subset \R^n$ be a bounded domain of $C^2$ class, and let $\mu \ge 2$ be fixed. We consider the Schr{\"o}dinger-type equation
\begin{equation}\label{linear eq ex1}
u - d(x)^{\mu} \Delta u =0 \quad \text{in $\Omega$,}
\end{equation}
where $d(x)$ denotes the distance from $x\in \Oba$ to $\partial \Omega$.
It is clearly a special case of \eqref{nonlinear eq} with $F(x, p, X)= - d(x)^{\mu}\tr X$. 
 It is easily seen that $F$ satisfies the assumptions (F1)--(F4). By applying Lemma \ref{lem f5} with $\sigma(x)=d(x)^{\mu\over 2} I_n$ and $\psi(x)=0$, we see that the assumption (F5) also holds in this case.

The condition $\mu\ge2$ is optimal. In fact, for $0<\mu<2$ and $\Omega= B_1$ it is known that there exist infinitely many bounded classical solutions of \eqref{linear eq ex1}. See \cite[Example~5.2]{BiPu}. 
\end{example}

We next revisit \cite[Example~5.3]{BiPu} with our approach. 
\begin{example}[Laplace-type equation with a drift term]\label{ex2}
Let $B_1 \subset \R^n$ be the ball centered at $0$ with radius $r=1$. Let $\mu \ge 2$ and fix $0\le \tau \le \mu -1$. We choose a vector-valued function $b=(b_1, b_2, \cdots, b_n)\in C^1(B_1, \R^n)$ such that 
\[
b(x)= d(x)^{-\tau}\frac{x}{|x|} \quad \text{for $\frac{3}{4}<|x|<1$},
\]
and consider the Schr{\"o}dinger-type equation
\begin{equation}\label{linear eq ex2}
u - d(x)^{\mu} \Delta u - d(x)^{\mu} \la b(x), \nabla u\ra =0 \quad \text{in $B_1$.}
\end{equation}
It is a special case of \eqref{nonlinear eq} with 
\[
F(x, p, X)= - d(x)^{\mu}\tr X - d(x)^{\mu}\la b(x), p \ra.
\]
It is easy to see that $F$ satisfies the assumptions (F1), (F2) and (F4). It is not difficult to show (F3) in this case, since $0\leq \tau\leq \mu-1$.  
The assumption (F5) can be verified by using Lemma \ref{lem f5} with $\sigma(x)= d(x)^{\mu\over 2}I_n$ and $\psi(x)=d(x)^\mu b(x)$. It is clear that $F$ satisfies \eqref{decompose3}. Note also that, due to $\mu\geq \tau+1$, 
\[
|\psi(x)-\psi(y)|=|d(x)^\mu b(x) - d(y)^\mu b(y)|\le C|x-y|
\]
for some $C>0$. We thus obtain \eqref{decompose1} and \eqref{decompose2}. 

Hence, by Theorem~\ref{thm main2}, we see that $u\equiv 0$ is the only bounded continuous viscosity solution to Schr{\"o}dinger-type equation \eqref{linear eq ex2} when $\mu\geq 2$ and $0\leq \tau\leq\mu-1$. It is shown in \cite[Example~5.3]{BiPu} that there exist infinitely many bounded classical solutions of \eqref{linear eq ex2} when $\tau>\mu -1$. Hence, our result is optimal in this case too. 
\end{example}

We can easily extend our application in the previous example to a general class of fully nonlinear elliptic equations of Hamilton-Jacobi-Isaacs type.

\begin{example}[Second order Hamilton-Jacobi-Isaacs equations]\label{ex3}
Assume that $\Omega\subset \R^n$ is a bounded domain of $C^2$ class.  
Let us look into the following fully nonlinear elliptic equation
\begin{equation}\label{hji eq}
    u-\inf_{\theta \in \Theta} \sup_{\lambda \in \Lambda}\,\left\{\tr (\Phi_{\lambda,\theta} (x)\nabla^2 u)+ \la b_{\lambda, \theta} (x), \nabla u \ra\right\} =0 \quad \text{in $\Omega$,}
\end{equation}
where $\Lambda$ and $\Theta$ are  index sets, and $\Phi_{\lambda, \theta}: \Omega\to \S^n$ and $b_{\lambda,\theta}: \Omega\to \R^n$ are continuous for each $\lambda\in \Lambda$ and $\theta \in \Theta$. Assume that for all $x\in \Omega$, $\lambda \in \Lambda$ and $\theta \in \Theta$, 
\[
\Phi_{\lambda, \theta}(x)=\sigma_{\lambda, \theta}(x)\sigma_{\lambda, \theta}(x)^T
\]
for some $\sigma_{\lambda, \theta}(x)\in \R^{n\times m}$ with a positive integer $m$. We thus see that 
\begin{equation}\label{hji-op}
F(x, p, X)=-\inf_{\theta \in \Theta} \sup_{\lambda \in \Lambda}\,\left\{\tr (\Phi_{\lambda,\theta} (x)X)+ \la b_{\lambda, \theta} (x), p \ra \right\} 
\end{equation}
is elliptic in the sense of (F2). 
We further assume that $\sigma_{\lambda, \theta}$ and $b_{\lambda, \theta}$ satisfy \eqref{decompose1} and \eqref{decompose2} with $L_1$ independent of $\lambda$ and $\theta$. Then it is easily seen that (F1), (F2) and (F4) hold in this case too.

Adopting the same arguments in Example \ref{ex2}, for each $\lambda\in \Lambda$ and $\theta\in \Theta$, we can verify all assumptions in Lemma \ref{lem f5} for 
\begin{equation}\label{isaccs1}
F_{\lambda, \theta}(x, p, X)=-\tr (\Phi_{\lambda,\theta} (x)X)- \la b_{\lambda, \theta} (x), p \ra
\end{equation}
 and use Lemma \ref{lem f5} to get 
\[
F_{\lambda, \theta}(y, q, Y)-F_{\lambda, \theta}(x, p, X)\leq \beta a_F |x-y|^\alpha+L\delta.
\]
 Since the choices of $\alpha, \beta, a_F, L$ are independent of $\lambda$ and $\theta$, by taking the supremum over $\theta\in \Theta$ and infimum over $\lambda\in \Lambda$, we immediately prove (F5).  
 \end{example}

\begin{example}[Parabolic Hamilton-Jacobi-Isaacs equations]\label{ex4}
    We can similarly obtain a Liouville-type result for parabolic equations of Hamilton-Jacobi-Isaacs type like
    \begin{equation}\label{parabolic hji eq}
    u_t+u-\inf_{\theta \in \Theta} \sup_{\lambda \in \Lambda}\,\left\{\tr (\Phi_{\lambda,\theta} (x)\nabla^2 u)+ \la b_{\lambda, \theta} (x), \nabla u \ra\right\} =0 \quad \text{in $\Omega\times (0, T)$,}
\end{equation}
where $\Omega$ is a $C^2$ bounded domain in $\R^n$, $\Phi_{\lambda, \theta}$ and $b_{\lambda, \theta}$ are the same as in Example \ref{ex3}. In this case, (F6) also holds in addition to (F1) through (F5). In fact, since Lemma \ref{lem f5} holds with constants $L_1, L_2>0$,  by choosing
\[
b_F=(1+L_1)L_1L_2,\quad  \ell= 2L_1L_2 (1+L_1(1+D^2)),
\]
we can verify (F6) following the proof of Lemma \ref{lem f5} (with $\alpha$ replaced by $2$ and $\alpha a_F$ by $c$). Indeed, under the conditions \eqref{parabolic1} and \eqref{parabolic2}, we can show that the operator $F_{\lambda, \theta}$ given by \eqref{isaccs1} satisfies
\[
F_{\lambda, \theta}(y, q, Y)-F_{\lambda, \theta}(x, p, X)\leq b_F c |x-y|^2 +\ell \delta,
\]
which yields \eqref{parabolic uc-cond}.

Then our result in Theorem \ref{thm main3} implies that $u\equiv 0$ is the only bounded viscosity solution $u\in C(\Omega\times [0, T))$ of \eqref{parabolic hji eq} that satisfies $u(\cdot, t)\to 0$ uniformly in $\Oba$ as $t\to 0$. It is clear that \eqref{parabolic hji eq} includes as a special case of the linear equation 
\[
u_t+u-d(x)^\mu\Delta u-d(x)^{\tau} \la b(x), \nabla u\ra=0 \quad \text{in $\Omega\times (0, T)$}
\]
with $\mu\geq 2$, $\tau\geq 1$ and $b\in C^1(\Oba)$. 
\end{example}

Besides the elliptic and parabolic problems, we can also apply our results to first order nonlinear equations. 
\begin{example}[First order Hamilton-Jacobi equations]
Let $\Omega\subset \R^n$ be a $C^2$ bounded domain. Consider the first order Hamilton-Jacobi equation
\begin{equation}\label{hj-eq}
    u+a(x)|\nabla u|^m=0 \quad\text{in $\Omega$},
\end{equation}
where $m> 1$ is a given exponent, and $a: \Omega\to [0, \infty)$ is a bounded nonnegative function satisfying
\begin{equation}\label{hj-lip}
    |a(x)^{1\over m}-a(y)^{1\over m}|\leq c|x-y| \quad \text{for all $x, y\in \Omega$}
\end{equation}
and
\begin{equation}\label{hj-decay}
    \sup_{x\in \Omega} {a(x)\over d(x)^\mu} \leq c 
\end{equation}
for some $c>0$ and $\mu\geq m$.  The case $m=1$ has essentially been discussed in Example \ref{ex3}.

The operator $F$ associated to \eqref{hj-eq} is of the form 
\[
F(x, p, X)=a(x)|p|^m, \quad \text{$x\in \Omega, p\in \R^n, X\in \S^n$.}
\]
Since $F$ does not depend on $X$, we simply write $F(x, p)$ below. 

One can easily verify the assumptions (F1)--(F4). Let us still focus on (F5). This time we fix $a_F>0$ and choose the constants $\alpha, \beta, L$ according to $M>0$. Take arbitrarily $x, y\in \Omega$ with $x\neq y$ and $p, q\in \R^n$ such that \eqref{uc-bounds} and \eqref{uc-cond-gradient} hold. By \eqref{uc-bounds}, we have 
\begin{equation}\label{hj-bound}
0\leq a(x)|p|^m\leq M, 
\end{equation}
On the other hand, by \eqref{uc-cond-gradient} and \eqref{hj-decay}, we get
\[
\begin{aligned}
|a(x)^{1\over m}p- \alpha a_F a(x)^{1\over m}|x-y|^{\alpha-2}(x-y)|\leq c^{1\over m}\delta d(x)^{{\mu\over m}-1},\\
|a(y)^{1\over m}q- \alpha a_F a(y)^{1\over m}|x-y|^{\alpha-2}(x-y)|\leq c^{1\over m}\delta d(y)^{{\mu\over m}-1}.\\
\end{aligned}
\]
It then follows from \eqref{hj-lip} that 
\begin{equation}\label{hj-eq1}
|a(x)^{1\over m}|p|-a(y)^{1\over m}|q|| \leq |a(x)^{1\over m}p-a(y)^{1\over m}q|\leq \alpha c a_F|x-y|^{\alpha} +2c^{1\over m} \delta D^{{\mu\over m}-1},
\end{equation}
where we still take $D=\max_{\Oba}d$. Applying the bound \eqref{hj-bound}, we can further obtain 
\[
|a(x)|p|^m-a(y)|q|^m|\leq \tilde{C} |a(x)^{1\over m}|p|-a(y)^{1\over m}|q||,
\]
where $\tilde{C}>0$ is a constant depending on $c, a_F, m, M$. Hence, using \eqref{hj-eq1} again, we deduce that
\[
|a(x)|p|^m-a(y)|q|^m|\leq \alpha c \tilde{C} a_F|x-y|^{\alpha} +2c^{1\over m} \tilde{C} \delta D^{{\mu\over m}-1},
\]
which implies the desired inequality
\[
F(y, q)-F(x, p)\leq \beta a_F|x-y|^\alpha +L \delta
\]
if we choose $L=2c^{1\over m} \tilde{C}D^{{\mu\over m}-1}$ and $\alpha\in (0, 1)$ small enough so that $\beta=\alpha c \tilde{C}<1$.

We conclude this example with a remark on the optimality of the condition $m\le \mu$. Let us consider a special case of the equation \eqref{hj-eq} with $n=1$, $a(x)=d(x)^\mu$, $\Omega=(0,2)$ and $m, \mu>0$, that is, 
\begin{equation}\label{first speical eq}
u+d(x)^\mu |u'|^m=0 \quad \text{in $(0,2)$}.
\end{equation}
 Since any continuous solution $u$ to \eqref{hj-eq} is nonpositive, we put $v=-u$ and consider the equation
\begin{equation}
v=x^\mu |v'|^m \quad \text{in $(0,\infty)$}.
\end{equation}
Let $a, b>0$. Then the solution $v$ satisfying $v(a)=b$ can be expressed locally as 
\begin{equation}
v(x)=
\begin{dcases}
\ \left[\pm \frac{m-1}{m-\mu}\left(x^{\frac{m-\mu}{m}}- a^{\frac{m-\mu}{m}}\right) + b^{\frac{m-1}{m}}\right]^{\frac{m}{m-1}},& \quad \text{if $m\neq 1,\, m\neq \mu$,}\\
\ b\exp \left( \pm\frac{ 1}{1-\mu}\left(x^{1-\mu}-a^{1-\mu}\right)\right),& \quad \text{if $m= 1,\, m\neq \mu$},\\
\ \left[\pm \frac{m-1}{m}\left(\log x -\log a \right) + b^{\frac{m-1}{m}}\right]^{\frac{m}{m-1}},& \quad \text{if $m=\mu\neq 1$, }\\
\ b\left(\frac{x}{a}\right)^{\pm 1},& \quad \text{if $m= \mu=1$.}    
\end{dcases}
\end{equation}
Based on this formula for $v$, assuming $\mu<m$, we obtain the following expression of $u$ for different ranges of $m$.
\begin{itemize}
\item If $m>1$, then
\[
u(x)=
\begin{dcases}
-\left(\frac{m-1}{m-\mu}\right)^{\frac{m}{m-1}}d(x)^{\frac{m-\mu}{m-1}},& \quad \text{for $0<x\le 1$},\\
-\left(\frac{m-1}{m-\mu}\right)^{\frac{m}{m-1}}\left[2-d(x)^{\frac{m-\mu}{m}}\right]^{\frac{m}{m-1}},& \quad \text{for $1<x<2$.}
\end{dcases}
\]
\item If $m<1$, then
\[
u(x)=
\begin{dcases}
-\left(\frac{m-\mu}{1-m}\right)^{\frac{m}{1-m}}\left[3-d(x)^{\frac{m-\mu}{m}}\right]^{-\frac{m}{1-m}},& \quad \text{for $0<x\le 1$},\\
-\left(\frac{m-\mu}{1-m}\right)^{\frac{m}{1-m}}\left[1+d(x)^{\frac{m-\mu}{m}}\right]^{-\frac{m}{1-m}},& \quad \text{for $1<x<2$.}
\end{dcases}
\quad 
\]
\item If $m=1$, then
\[
u(x)=
\begin{dcases}
-\exp \left(\frac{d(x)^{1-\mu}}{1-\mu}\right),& \quad \text{for $0<x\le 1$},\\
-\exp \left(\frac{2-d(x)^{1-\mu}}{1-\mu}\right),& \quad \text{for $1<x<2$.}
\end{dcases}
\]
\end{itemize}
These are bounded classical solutions of \eqref{first speical eq}. 
Therefore, the condition $m\le \mu$ is necessary to ensure that there are only trivial solutions. Such nontrivial bounded solutions $u$ demonstrate the boundary behavior $|u'(x)|\to \infty$ as $x\to 0+$ or $x\to 2-$, which is related to the singular boundary conditions studied in \cite{LaL2}. 
\end{example}

Our examples above suggest that the degeneracy rate of nonlinear elliptic or parabolic operators to guarantee nonexistence of nontrivial solutions is determined by the order of the equation. As a future problem, it would be interesting to discuss the critical exponent $\mu$ for degenerate fractional Laplacian $d(x)^\mu(-\Delta)^{s}$ for $0<s<1$ and investigate whether the same scenario extends to nonlocal operators. 


\end{document}